\documentclass[12pt, reqno]{amsart}
\usepackage{amsmath, amsthm, amscd, amsfonts, amssymb, graphicx, color}
\usepackage[bookmarksnumbered, colorlinks, plainpages]{hyperref}

\newtheorem{theorem}{Theorem}[section]

\newtheorem{proposition}[theorem]{Proposition}

\theoremstyle{definition}

\newtheorem{example}[theorem]{Example}

\theoremstyle{remark}

\numberwithin{equation}{section}

\begin{document}
\setcounter{page}{1}


\title[The k-rank numerical radii]{The k-rank numerical radii}

\author[Aik. Aretaki, J. Maroulas]{Aikaterini Aretaki$^1$ and John Maroulas$^2$$^{*}$}

\address{$^{1}$ National Technical University of Athens, Department of Mathematics, Zografou Campus, Athens 15780, Greece.}
\email{\textcolor[rgb]{0.00,0.00,0.84}{kathy@mail.ntua.gr}}

\address{$^{2}$ National Technical University of Athens, Department of Mathematics, Zografou Campus, Athens 15780, Greece.}
\email{\textcolor[rgb]{0.00,0.00,0.84}{maroulas@math.ntua.gr}}




\date{Received: 14 December 2011; Revised: 13 February 2012; Accepted: zzzzzz.
\newline \indent $^{*}$ Corresponding author}

\begin{abstract}
The $k$-rank numerical range $\Lambda_{k}(A)$ is expressed via an inter\-section of any countable family of numerical ranges $\{F(M^{*}_{\nu}AM_{\nu})\}_{\nu\in\mathbb{N}}$ with respect to $n\times (n-k+1)$ isometries $M_{\nu}$. This implication for $\Lambda_{k}(A)$ provides further ela\-boration of the $k$-rank numerical radii of $A$.
\end{abstract} \maketitle

\section{Introduction}\label{intro}

\noindent Let $\mathcal{M}_{n}(\mathbb{C})$ be the algebra of $n\times n$ complex matrices and $k\geq 1$ be a positive integer.
The \emph{k-rank numerical range} $\Lambda_{k}(A)$ of a matrix $A\in\mathcal{M}_{n}$ is defined  by
\begin{eqnarray*}
\nonumber\Lambda_{k}(A) & = & \{\lambda\in\mathbb{C} : X^{*}AX=\lambda I_{k}\,\, \textrm{for}\,\, \textrm{some}\,\, X\in\mathcal{X}_{k}\} \\
 & = & \{\lambda\in\mathbb{C} : PAP=\lambda P\,\, \mathrm{for}\,\, \mathrm{some}\,\, P\in\mathcal{Y}_{k}\},
\end{eqnarray*}
where $\mathcal{X}_{k}=\{X\in\mathcal{M}_{n,k}:\,\,X^{*}X=I_{k}\}$ and $\mathcal{Y}_{k}=\{P\in\mathcal{M}_{n}: P=XX^{*}, X\in\mathcal{X}_{k}\}$. Note that $\Lambda_{k}(A)$ has been introduced  as a versatile tool to solving a fundamental error correction problem in quantum computing \cite{Aretaki,Aret,Poon-Li-Sze,Li-Sze,Hugo}.

For $k=1$, $\Lambda_{k}(A)$ reduces to the classical  \emph{numerical range} of a matrix $A$,
\[
\Lambda_{1}(A)\equiv F(A)=\{ x^{*}Ax : x\in \mathbb{C}^{n}, \,x^{*}x=1\},
\]
which is known to be a compact and convex subset of $\mathbb{C}$ \cite{H.J.T}, as well as the same properties hold for the set $\Lambda_{k}(A)$, for $k>1$ \cite{Li-Sze,Hugo}. Associated with $\Lambda_{k}(A)$ are the \textit{k-rank numerical radius} $r_{k}(A)$ and the \emph{inner k-rank numerical radius} $\widetilde{r}_{k}(A)$, defined respectively, by
\[
r_{k}(A)=\max{\{|z|: z\in\partial\Lambda_{k}(A)\}}\,\,\,\,\textrm{and}\,\,\,\,\widetilde{r}_{k}(A)=\min{\{|z|: z\in\partial\Lambda_{k}(A)\}}.
\]
For $k=1$, they yield the \textit{numerical radius} and the \emph{inner numerical radius},
\[
r(A)=\max{\{|z|: z\in\partial F(A)\}}\,\,\,\,\textrm{and}\,\,\,\,\widetilde{r}(A)=\min{\{|z|: z\in\partial F(A)\}},
\]
respectively.

In the first section of this paper,  $\Lambda_{k}(A)$ is  proved to coincide with an inde\-finite intersection of  numerical ranges of all the compressions of $A\in\mathcal{M}_{n}$ to $(n-k+1)$-dimensional subspaces, which has been also used in \cite{Aretaki,Aret}. Further elaboration led us to reformulate
$\Lambda_{k}(A)$ in terms of an intersection of a countable family of numerical ranges. This result provides  additional characterizations of   $r_{k}(A)$ and $\widetilde{r}_{k}(A)$, which are presented in section 3.

\section{Alternative expressions of $\Lambda_{k}(A)$}\label{sec:1}
Initially, the higher rank numerical range $\Lambda_{k}(A)$  is proved to be equal to an infinite intersection of numerical ranges.
\begin{theorem}\label{th1}
Let $A\in\mathcal{M}_{n}(\mathbb{C})$. Then
\[
\Lambda_{k}(A)=\bigcap_{M\in\mathcal{X}_{n-k+1}}{F(M^{*}AM)}=\bigcap_{P\in\mathcal{Y}_{n-k+1}}{F(PAP)}.
\]
\end{theorem}
\begin{proof}
Denoting by $\lambda_{1}(H)\geq \ldots \geq \lambda_{n}(H)$ the decreasingly ordered eigenvalues of a hermitian
matrix $H\in\mathcal{M}_{n}(\mathbb{C})$, we have \cite{Li-Sze}
\[
\Lambda_{k}(A)=\bigcap_{\theta\in[0, 2\pi)}{e^{-\mathrm{i}\theta}\{z\in\mathbb{C} : \mathrm{Re} z \leq \lambda_{k}(H(e^{\mathrm{i}\theta}A))\}}
\]
where $H(\cdot)$ is the hermitian part of a matrix. Moreover, by Courant-Fisher theorem, we have
\[
\lambda_{k}(H(e^{\mathrm{i}\theta}A))=\min_{\dim \mathcal{S}=n-k+1}{\max_{\substack{x\in \mathcal{S}\\ \|x\|=1}}{x^{*}H(e^{\mathrm{i}\theta}A)x}}.
\]
Denoting by $\mathcal{S}=span\{u_{1}, \ldots, u_{n-k+1}\}$, where $u_{i}\in\mathbb{C}^{n}$, $i=1, \ldots, n-k+1$ are orthonormal vectors,
then any unit vector $x\in\mathcal{S}$ is written in the form $x=My$, where
$M=\begin{bmatrix} u_{1} & \cdots & u_{n-k+1} \\ \end{bmatrix}\in\mathcal{X}_{n-k+1}$ and $y\in\mathbb{C}^{n-k+1}$ is unit.
Hence, we have
\begin{eqnarray*}
 \lambda_{k}(H(e^{\mathrm{i}\theta}A)) & = & \min_{M}{\max_{\substack{y\in \mathbb{C}^{n-k+1}\\ \|y\|=1}}{y^{*}M^{*}H(e^{\mathrm{i}\theta}A)My}}\\
    & = & \min_{M}{\max_{\substack{y\in \mathbb{C}^{n-k+1}\\ \|y\|=1}}{y^{*}H(e^{\mathrm{i}\theta}M^{*}AM)y}} \\
    & = & \min_{M}{\lambda_{1}(H(e^{\mathrm{i}\theta}M^{*}AM))}
\end{eqnarray*}
and consequently
\begin{eqnarray*}
\Lambda_{k}(A) & = & \bigcap_{\theta}{e^{-\mathrm{i}\theta}\{z\in\mathbb{C} : \mathrm{Re}z \leq \min_{M}{\lambda_{1}(H(e^{\mathrm{i}\theta}M^{*}AM))}\}}\\
               & = & \bigcap_{M}{\bigcap_{\theta}{e^{-\mathrm{i}\theta}\{z\in\mathbb{C} : \mathrm{Re}z \leq \lambda_{1}(H(e^{\mathrm{i}\theta}M^{*}AM))\}}}\\
               & = & \bigcap_{M\in\mathcal{X}_{n-k+1}}{F(M^{*}AM)}.
\end{eqnarray*}
Moreover, if we consider the $(n-k+1)$-rank orthogonal projection $P=MM^{*}$ of $\mathbb{C}^{n}$ onto the aforementioned space $\mathcal{S}$, then $x=Px$, for $x\in\mathcal{S}$ and $P\hat{x}=0$, for $\hat{x}\notin\mathcal{S}$. Hence, we have
\[
\Lambda_{k}(A)=\bigcap_{P\in\mathcal{Y}_{n-k+1}}{F(PAP)}.
\]
\end{proof}
At this point, we should note that Theorem \ref{th1} provides a different and independent characterization of $\Lambda_{k}(A)$ than the one given in \cite[Cor. 4.9]{Poon-Li-Sze}. We focus on the expression of $\Lambda_{k}(A)$ via the numerical ranges $F(M^{*}AM)$ (or $F(PAP)$), since it represents a more useful and adva\-ntageous procedure to determine and approximate  the boundary of $\Lambda_{k}(A)$ numerically.

In addition, Theorem \ref{th1}  verifies the ``\emph{convexity of $\Lambda_{k}(A)$}''  through the convexi\-ty of the numerical ranges $F(M^{*}AM)$ (or $F(PAP)$), which is ensured by the Toeplitz-Hausdorff theorem. A different way of indicating that $\Lambda_{k}(A)$ is convex, is developed in \cite{Hugo}. For $k=n$, clearly $\Lambda_{n}(A)=\bigcap_{x\in\mathbb{C}^{n}, \|x\|=1}F(x^{*}Ax)$ and should be $\Lambda_{n}(A)\neq\emptyset$
\emph{precisely} when $A$ is scalar.

Motivated by the above, we present the main result of our paper, redescribing the higher rank numeri\-cal range as  \textit{a countable intersection of numerical ranges}.
\begin{theorem}\label{th2}
Let $A\in\mathcal{M}_{n}$. Then for any countable family of orthogonal projections $\{P_{\nu}: \nu\in\mathbb{N}\}\subseteq\mathcal{Y}_{n-k+1}$
(or any family of isometries $\{M_{\nu}: \nu\in\mathbb{N}\}\subseteq\mathcal{X}_{n-k+1}$) we have
\begin{equation}\label{eq2.1}
\Lambda_{k}(A)=\bigcap_{\nu\in\mathbb{N}}F(P_{\nu}AP_{\nu})=\bigcap_{\nu\in\mathbb{N}}F(M_{\nu}^{*}AM_{\nu}).
\end{equation}
\end{theorem}
\begin{proof}
By Theorem \ref{th1}, we have
\[
[\Lambda_{k}(A)]^{c}=\mathbb{C}\setminus\Lambda_{k}(A)=\bigcup_{P\in\mathcal{Y}_{n-k+1}}[F(PAP)^{c}],
\]
whereupon the family  $\{F(PAP)^{c}: P\in\mathcal{Y}_{n-k+1}\}$ is an open cover of $[\Lambda_{k}(A)]^{c}$.
Moreover, $[\Lambda_{k}(A)]^{c}$ is separable, as an open  subset of the sepa\-rable space $\mathbb{C}$ and then $[\Lambda_{k}(A)]^{c}$ has a countable base \cite{top}, which obviously depends on the matrix $A$. This fact guarantees that any open cover of $[\Lambda_{k}(A)]^{c}$ admits a countable subcover, leading to the relation
\[
[\Lambda_{k}(A)]^{c}=\bigcup_{\nu\in\mathbb{N}}[F(P_{\nu}AP_{\nu})^{c}],
\]
i.e. leading to the first equality in \eqref{eq2.1}. Taking into consideration that there exists a countable dense subset $\mathcal{J}\subseteq\mathcal{Y}_{n-k+1}$ with respect to the operator norm $\|\cdot\|$ and $P_{\nu}\in\mathcal{Y}_{n-k+1}$, for $\nu\in\mathbb{N}$, clearly, $\bigcap_{\nu\in\mathbb{N}}F(P_{\nu}AP_{\nu})=\bigcap_{\nu\in\mathbb{N},P_{\nu}\in \mathcal{J}}F(P_{\nu}AP_{\nu})$.
That is in \eqref{eq2.1}, the family of orthogonal projections $\{P_{\nu}: \nu\in\mathbb{N}\}$ can be chosen independently of $A$. Moreover, due to $P_{\nu}=M_{\nu}M^{*}_{\nu}$, with $M_{\nu}\in\mathcal{X}_{n-k+1}$, we derive the second equality in \eqref{eq2.1}.
\end{proof}
For a construction of a countable family of isometries $\{M_{\nu}: \nu\in\mathbb{N}\}\subseteq\mathcal{X}_{n-k+1}$, see also in the Appendix.
\\\\
Furthermore, using the dual ``max-min'' expression of the $k$-th eigenvalue,
\[
\lambda_{k}(H(e^{\mathrm{i}\theta}A))=\max_{\dim \mathcal{G}=k}{\min_{\substack{x\in\mathcal{G}\\ \|x\|=1}}{x^{*}H(e^{\mathrm{i}\theta}A)x}}=
\max_{N}{\lambda_{\min}(H(e^{\mathrm{i}\theta}N^{*}AN))},
\]
where $N\in\mathcal{X}_{k}$, we have
\begin{eqnarray}\label{t}
\nonumber\Lambda_{k}(A) & = & \bigcap_{\theta}{e^{-\mathrm{i}\theta}\{z\in\mathbb{C} : \mathrm{Re}z \leq \max_{N}{\lambda_{k}(H(e^{\mathrm{i}\theta}N^{*}AN))}\}}\\
\nonumber     & = & \bigcup_{N}{\bigcap_{\theta}{e^{-\mathrm{i}\theta}\{z\in\mathbb{C} : \mathrm{Re}z \leq \lambda_{k}(H(e^{\mathrm{i}\theta}N^{*}AN))\}}}\\
              & = & \bigcup_{N\in\mathcal{X}_{k}}{\Lambda_{k}(N^{*}AN)},
\end{eqnarray}
and due to the convexity of $\Lambda_{k}(A)$, we establish
\begin{equation}\label{co}
\Lambda_{k}(A)=\mathrm{co}\bigcup_{N\in\mathcal{X}_{k}}{\Lambda_{k}(N^{*}AN)},
\end{equation}
where $\mathrm{co}(\cdot)$ denotes the convex hull of a set. Apparently, $\Lambda_{k}(N^{*}AN)\neq\emptyset$ if and only if
$N^{*}AN=\lambda I_{k}$ \cite{Poon-Li-Sze} and then \eqref{co} is reduced to $\bigcup_{N}{\Lambda_{k}(N^{*}AN)}=\bigcup_{N}{\{\lambda: N^{*}AN=\lambda I_{k}\}}=\Lambda_{k}(A)$, where $N$ runs all $n\times k$ isometries.

In spite of Theorem \ref{th2}, $\Lambda_{k}(A)$ cannot be described as a countable union in \eqref{t}, because if
\[
\Lambda_{k}(A)=\bigcup_{\nu\in\mathbb{N}}\{\Lambda_{k}(N^{*}_{\nu}AN_{\nu}): N_{\nu}\in\mathcal{X}_{k}\}=
\bigcup_{\nu\in\mathbb{N}}\{\lambda_{\nu}:N^{*}_{\nu}AN_{\nu}=\lambda_{\nu}I_{k},\,N_{\nu}\in\mathcal{X}_{k}\},
\]
then $\Lambda_{k}(A)$ should be a countable set, which is not true.

\section{Properties of $r_{k}(A)$ and $\widetilde{r}_{k}(A)$}
In this section,  we characterize the $k$-rank numerical radius $r_{k}(A)$ and the inner $k$-rank numerical radius $\widetilde{r}_{k}(A)$. Motivated by Theorem \ref{th2}, we present the next two results.

\begin{theorem}\label{th3}
Let $A\in\mathcal{M}_{n}$ and $\mathcal{J}_{\nu}(A)=\bigcap_{p=1}^{\nu}F(M_{p}^{*}AM_{p})$, where $M_{p}\in\mathcal{X}_{n-k+1}$. Then
\begin{equation*}
r_{k}(A)=\lim_{\nu\to\infty}\sup\{|z| : z\in\mathcal{J}_{\nu}(A)\}
=\inf_{\nu\in\mathbb{N}}\sup\{|z| : z\in\mathcal{J}_{\nu}(A)\}.
\end{equation*}
\end{theorem}
\begin{proof}
By Theorem \ref{th2}, we have
\begin{equation}\label{rel7}
\Lambda_{k}(A)=\bigcap_{\nu=1}^{\infty}\mathcal{J}_{\nu}(A)\subseteq\mathcal{J}_{\nu}(A)\subseteq F(A)
\subseteq\mathcal{D}(0,\|A\|_{2}),
\end{equation}
for all $\nu\in\mathbb{N}$, where the sequence $\{\mathcal{J}_{\nu}(A)\}_{\nu\in\mathbb{N}}$ is nonincreasing and $\mathcal{D}(0,\|A\|_{2})$ is the circular disc centered at the origin with radius the spectral norm $\|A\|_{2}$ of $A\in\mathcal{M}_{n}$. Clearly,
\[
r_{k}(A)=\max_{z\in\bigcap_{\nu=1}^{\infty}\mathcal{J}_{\nu}(A)}|z|\leq
\sup_{z\in\mathcal{J}_{\nu}(A)}|z|\leq r(A)\leq\|A\|_{2},
\]
then the nonincreasing and bounded sequence $q_{\nu}=\sup\{|z|: z\in\mathcal{J}_{\nu}(A)\}$  converges. Therefore
\[
r_{k}(A)\leq\lim_{\nu\to\infty}q_{\nu}=q_{0}.
\]
We shall prove that the above inequality is actually an equality. Assume  that $r_{k}(A)<q_{0}$. In this case, there is $\varepsilon>0$, where $r_{k}(A)+\varepsilon<q_{0}\leq q_{\nu}$ for all $\nu\in\mathbb{N}$. Then we may find a sequence  $\{\zeta_{\nu}\}\subseteq\mathcal{J}_{\nu}(A)$ such that $q_{0}\leq |\zeta_{\nu}|$ for all $\nu\in\mathbb{N}$. Due to the boundedness of the set $\mathcal{J}_{\nu}(A)$, the sequence $\{\zeta_{\nu}\}$ contains a subsequence $\{\zeta_{\rho_{\nu}}\}$ converging to $\zeta_{0}\in\mathbb{C}$ and clearly, we obtain  $q_{0}\leq |\zeta_{0}|$. Because of the monotonicity of $\mathcal{J}_{\nu}(A)$ (i.e. $\mathcal{J}_{\nu+1}(A)\subseteq\mathcal{J}_{\nu}(A)$), $\zeta_{\rho_{\nu}}$ eventually belong to $\mathcal{J}_{\nu}(A)$,\, $\forall\,\,\nu\in\mathbb{N}$, meaning that    $\{\zeta_{\rho_{\nu}}\}\subseteq\bigcap_{\nu=1}^{\infty}\mathcal{J}_{\nu}(A)=\Lambda_{k}(A)$ and since $\Lambda_{k}(A)$ is closed, $\zeta_{0}\in\Lambda_{k}(A)$. It implies $|\zeta_{0}|\leq r_{k}(A)$ and then $q_{0}\leq r_{k}(A)$, a contradiction.

The second equality is apparent.
\end{proof}
\begin{theorem}\label{th4}
Let $A\in\mathcal{M}_{n}$ and $\mathcal{J}_{\nu}(A)=\bigcap_{p=1}^{\nu}F(M_{p}^{*}AM_{p})$, for some $M_{p}\in\mathcal{X}_{n-k+1}$. If\, $0\notin\Lambda_{k}(A)$, then
\begin{equation*}
\widetilde{r}_{k}(A)=\lim_{\nu\to\infty}\inf\{|z|: z\in\mathcal{J}_{\nu}(A)\}=\sup_{\nu\in\mathbb{N}}\inf\{|z|: z\in\mathcal{J}_{\nu}(A)\}.
\end{equation*}
\end{theorem}
\begin{proof}
Obviously, $0\notin\Lambda_{k}(A)$ indicates $\widetilde{r}_{k}(A)=\min\{|z|: z\in\Lambda_{k}(A)\}$ and by the relation \eqref{rel7}, it is clear that
\[
\|A\|_{2}\geq r(A)\geq\widetilde{r}_{k}(A)=\min_{z\in\bigcap_{\nu=1}^{\infty}\mathcal{J}_{\nu}(A)} |z|\geq
\inf_{z\in\mathcal{J}_{\nu}(A)}|z|.
\]
Consequently, the sequence $t_{\nu}=\inf\{|z|: z\in\mathcal{J}_{\nu}(A)\}$, $\nu\in\mathbb{N}$, is  nondecreasing and bounded and we have
\[
\widetilde{r}_{k}(A)\geq\lim_{\nu\to\infty}t_{\nu}=t_{0}.
\]
In a similar way as in Theorem \ref{th3}, we will show that $\widetilde{r}_{k}(A)=\lim_{\nu\to\infty}t_{\nu}$. Suppose $\widetilde{r}_{k}(A)>t_{0}$, then $t_{\nu}\leq t_{0}<\widetilde{r}_{k}(A)-\varepsilon$, for all $\nu\in\mathbb{N}$ and $\varepsilon>0$. Considering the sequence
$\{\widetilde{\zeta}_{\nu}\}\subseteq\mathcal{J}_{\nu}(A)$ such that $|\widetilde{\zeta}_{\nu}|\leq t_{0}$, let its subsequence $\{\widetilde{\zeta}_{s_{\nu}}\}$  converging to $\widetilde{\zeta}_{0}$, with $|\widetilde{\zeta}_{0}|\leq t_{0}$. Since $\{\mathcal{J}_{\nu}(A)\}$ is nonincreasing, $\widetilde{\zeta}_{s_{\nu}}$ eventually belong to $\mathcal{J}_{\nu}(A)$,\, $\forall\,\, \nu\in\mathbb{N}$, establishing $\{\widetilde{\zeta}_{s_{\nu}}\}\subseteq\bigcap_{\nu\in\mathbb{N}}\mathcal{J}_{\nu}(A)=\Lambda_{k}(A)$. Hence, we conclude  $\widetilde{\zeta}_{0}\in\bigcap_{\nu=1}^{\infty}\mathcal{J}_{\nu}(A)=\Lambda_{k}(A)$, i.e.  $t_{0}\geq|\widetilde{\zeta}_{0}|\geq\widetilde{r}_{k}(A)$, absurd.

The second equality is trivial.
\end{proof}
The next proposition asserts  a lower and an upper bound for  $r_{k}(A)$ and $\widetilde{r}_{k}(A)$, respectively.
\begin{proposition}
Let $A\in\mathcal{M}_{n}$ and $M_{p}\in\mathcal{X}_{n-k+1}$, $p\in\mathbb{N}$, then
\[
r_{k}(A)\leq\inf_{p\in\mathbb{N}}r(M_{p}^{*}AM_{p}).
\]
If $0\notin\Lambda_{k}(A)$, then $$\widetilde{r}_{k}(A)\geq\inf_{p\in\mathbb{N}}\widetilde{r}(M^{*}_{p}AM_{p}).$$
\end{proposition}
\begin{proof}
By Theorem \ref{th2}, we obtain $\partial\Lambda_{k}(A)\subseteq\Lambda_{k}(A)\subseteq F(M_{p}^{*}AM_{p})$ for all $p\in\mathbb{N}$. Then
\[
r_{k}(A)=\max\{|z|: z\in\Lambda_{k}(A)\}\leq\max\{|z|: z\in F(M_{p}^{*}AM_{p})\}=r(M_{p}^{*}AM_{p}).
\]
Denoting by $c(M_{p}^{*}AM_{p})=\min\{|z|: z\in F(M_{p}^{*}AM_{p})\}$ for all $p\in\mathbb{N}$, we have
\[
\widetilde{r}_{k}(A)\geq\min\{|z|: z\in\Lambda_{k}(A)\}\geq c(M_{p}^{*}AM_{p}).
\]
Since $0\leq c(M_{p}^{*}AM_{p})\leq\widetilde{r}(M_{p}^{*}AM_{p})\leq r(M_{p}^{*}AM_{p})\leq\|A\|_{2}$ for any $p\in\mathbb{N}$, immediately, we obtain
\begin{equation*}
r_{k}(A)\leq\inf_{p\in\mathbb{N}}r(M_{p}^{*}AM_{p})\,\,\, \textrm{and}\,\,\,\,\widetilde{r}_{k}(A)\geq\sup_{p\in\mathbb{N}}c(M_{p}^{*}AM_{p}).
\end{equation*}
If $0\notin\Lambda_{k}(A)$, then by Theorem \ref{th2}, $0\notin F(M^{*}_{l}AM_{l})$ for some $l\in\mathbb{N}$, $M_{l}\in\mathcal{X}_{n-k+1}$ and $c(M^{*}_{l}AM_{l})=\widetilde{r}(M^{*}_{l}AM_{l})$. Hence
\[
\widetilde{r}_{k}(A)\geq\sup_{p\in\mathbb{N}}c(M^{*}_{p}AM_{p})\geq\widetilde{r}(M^{*}_{l}AM_{l})\geq\inf_{p\in\mathbb{N}}
\widetilde{r}(M^{*}_{p}AM_{p}).
\]
\end{proof}
The numerical radius function $r(\cdot):\mathcal{M}_{n}\to\mathbb{R}_{+}$ is not a matrix norm, never\-theless, it satisfies the power inequality $r(A^{m})\leq[r(A)]^{m}$, for all po\-sitive integers $m$, which is  utilized  for stability issues of several iterative methods \cite{Ando,H.J.T}. On the other hand, the $k$-rank numerical radius fails to satisfy the power inequality, as the next counterexample reveals.
\begin{example}
Let the matrix $A=\left[\begin{smallmatrix}
  1.8 & 2 & 3 & 4 \\
  0 & 0.8+\mathrm{i} & 0 & \mathrm{i} \\
 -2 & 1 & -1.2 & 1 \\
  0 & 0 & 1 & 0.8 \\
\end{smallmatrix}\right]$. Using Theorems \ref{th1} and \ref{th2}, the set $\Lambda_{2}(A)$ is illustrated in the left part of Figure \ref{fig1}  by the uncovered area inside the figure. Clearly, it is included in the unit circular disc, which indicates that $r_{2}(A)<1$. On the other hand, the set $\Lambda_{2}(A^{2})$, illustrated in the right part of  Figure \ref{fig1} with the same manner, is not bounded by the unit circle and thus $r_{2}(A^{2})>1$. Obviously, $[r_{2}(A)]^{2}<1<r_{2}(A^{2})$.\\

\begin{figure}[here]
\begin{tabular}{cc}
 \includegraphics[width=6cm]{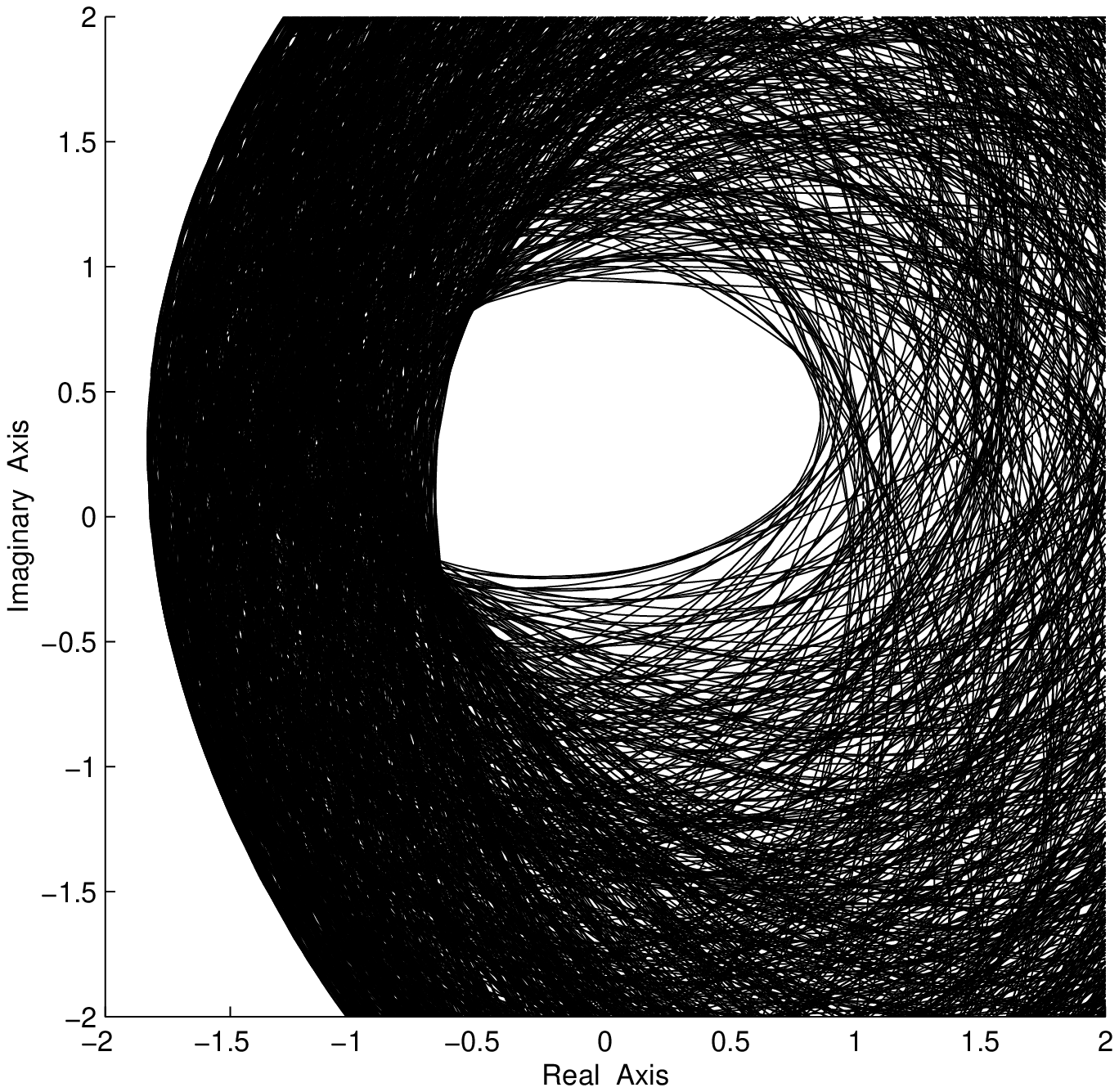} & \includegraphics[width=6cm]{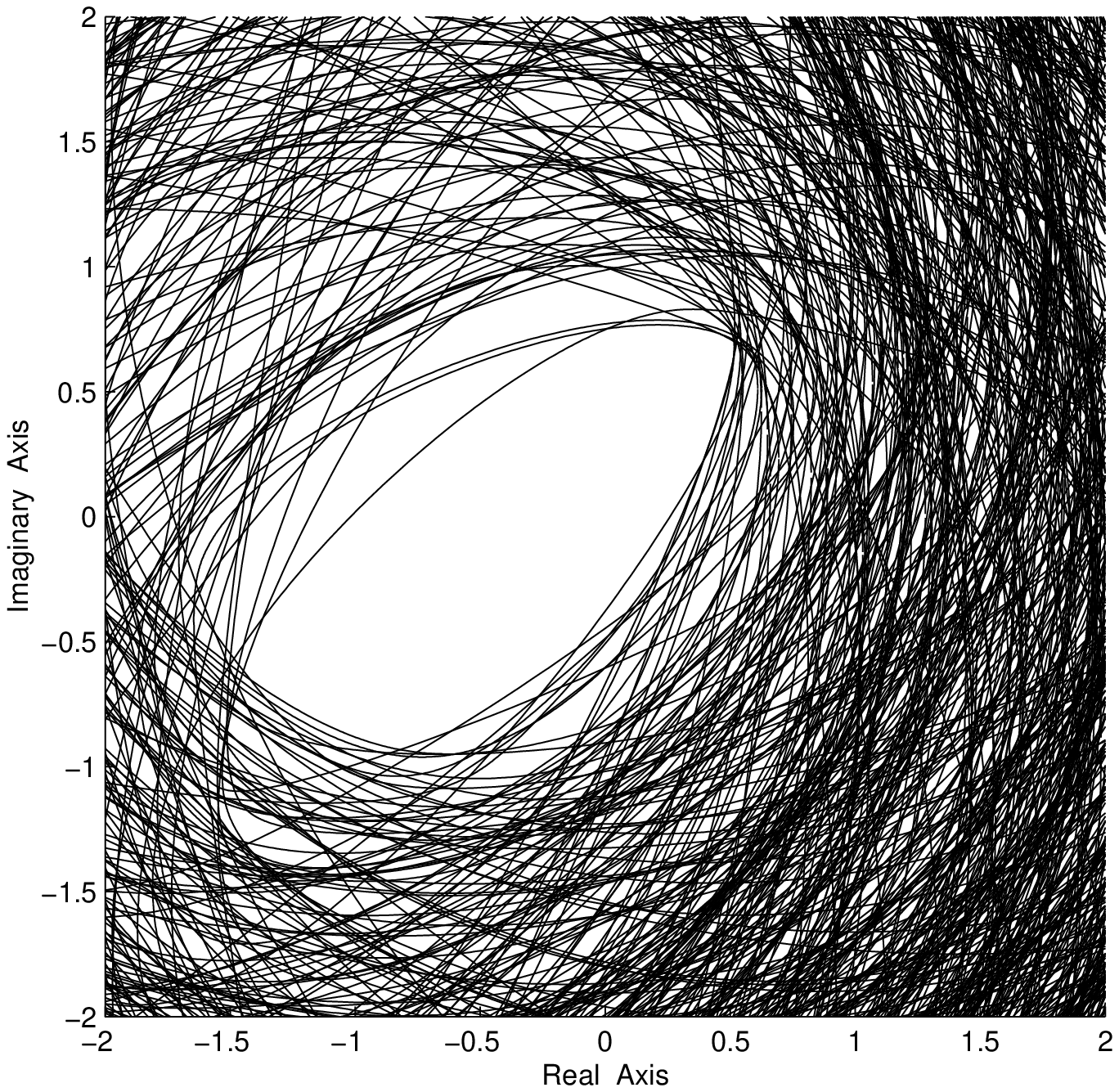} \\
\end{tabular}
\caption{The ``white'' bounded areas inside the figures depict the sets $\Lambda_{2}(A)$ (left) and  $\Lambda_{2}(A^{2})$ (right).}\label{fig1}
\end{figure}
\end{example}
\par
The results developed in this paper draw attention to the rank-$k$ nume\-rical range $\Lambda_{k}(L(\lambda))$ of a matrix polynomial  $L(\lambda)=\sum_{i=0}^{m}A_{i}\lambda^{i}$ $(A_{i}\in\mathcal{M}_{n})$, which has been  extensively studied in \cite{Aretaki,Aret}. It is worth noting that Theorem 2.2 can be also gene\-ralized in the case of $L(\lambda)$, which follows readily from the proof. Hence,  the rank-$k$ numerical radii of $\Lambda_{k}(L(\lambda))$ can be elaborated with the same spirit as here \cite{Adam}.

\appendix
\section{\null}

Following we provide \textit{another construction} of a family of $n\times(n-k+1)$ isometries $\{M_{\nu}: \nu\in\mathbb{N}\}$ presented in Theorem \ref{th2}.
\begin{proof}
By Theorem \ref{th1}, we have
\begin{equation}\label{eq0}
\Lambda_{k}(A)=\bigcap_{M\in\mathcal{X}_{n-k+1}}F(M^{*}AM),
\end{equation}
which is known to be a compact and convex subset  of $\mathbb{C}$. For any $n\times(n-k+1)$ isometry $M_{\nu}$ $(\nu\in\mathbb{N})$, we have $\Lambda_{k}(A)\subseteq F(M_{\nu}^{*}AM_{\nu})$ for all $\nu\in\mathbb{N}$ and thus,
\begin{equation}\label{eq1}
\Lambda_{k}(A)\subseteq\bigcap_{\nu\in\mathbb{N}}F(M_{\nu}^{*}AM_{\nu}).
\end{equation}
In order to prove equality in the relation \eqref{eq1}, we  distinguish two cases for the interior of $\Lambda_{k}(A)$.

Suppose first that $\mathrm{int}\Lambda_{k}(A)\neq\emptyset$. Then by \eqref{eq1}, we obtain
\[
\emptyset\neq \mathrm{int}\Lambda_{k}(A)\subseteq \mathrm{int}\bigcap_{\nu\in\mathbb{N}}F(M_{\nu}^{*}AM_{\nu})
\]
and since $\bigcap_{\nu}F(M_{\nu}^{*}AM_{\nu})$ is convex and closed, we establish
\begin{equation}\label{eq4}
\overline{\mathrm{int}\bigcap_{\nu\in\mathbb{N}}F(M_{\nu}^{*}AM_{\nu})}=\bigcap_{\nu\in\mathbb{N}}F(M_{\nu}^{*}AM_{\nu}),
\end{equation}
where $\overline{\,\,\,\cdot\,\,}$ denotes the closure of a set. Thus, combining the relations \eqref{eq1} and \eqref{eq4}, we have
\begin{equation}\label{eq5}
\Lambda_{k}(A)\subseteq\overline{\mathrm{int}\bigcap_{\nu\in\mathbb{N}}F(M_{\nu}^{*}AM_{\nu})}.
\end{equation}
Further, we claim that $\mathrm{int}\bigcap_{\nu}F(M_{\nu}^{*}AM_{\nu})\subseteq\Lambda_{k}(A)$. Assume on the contrary that $z_{0}\in \mathrm{int}\bigcap_{\nu}F(M_{\nu}^{*}AM_{\nu})$ but $z_{0}\notin\Lambda_{k}(A)$, then there exists an open neighborhood
$\mathcal{B}(z_{0},\varepsilon)$, with $\varepsilon >0$, such that
\begin{equation*}
\mathcal{B}(z_{0},\varepsilon)\subset\bigcap_{\nu\in\mathbb{N}}F(M_{\nu}^{*}AM_{\nu})\,\,\,\textrm{and}\,\,\,
\mathcal{B}(z_{0},\varepsilon)\cap\Lambda_{k}(A)=\emptyset.
\end{equation*}
Then, the set $[\Lambda_{k}(A)]^{c}=\mathbb{C}\setminus\Lambda_{k}(A)$ is separable, as an open  subset of the sepa\-rable space $\mathbb{C}$ and let $\mathcal{Z}$ be a countable dense subset of $[\Lambda_{k}(A)]^{c}$ \cite{top}.
Therefore, there exists a sequence $\{z_{p}:p\in\mathbb{N}\}$ in $\mathcal{Z}$ such that
$\lim_{p\to\infty}z_{p}=z_{0}$ and $z_{p}\in\mathcal{B}(z_{0},\varepsilon)$. Moreover, $z_{p}\in[\Lambda_{k}(A)]^{c}$ and by \eqref{eq0}, it follows that for any $p$ correspond indices $j_{p}\in\mathbb{N}$ such that $z_{p}\notin F(M_{j_{p}}^{*}AM_{j_{p}})$. Thus $z_{p}\notin\bigcap_{p\in\mathbb{N}}F(M_{j_{p}}^{*}AM_{j_{p}})$, which is absurd, since $z_{p}\in\mathcal{B}(z_{0},\varepsilon)\subset\bigcap_{\nu\in\mathbb{N}}F(M_{\nu}^{*}AM_{\nu})$. Hence $z_{0}\in\Lambda_{k}(A)$,
verifying our claim and we obtain
\begin{equation}\label{eq8}
\overline{\mathrm{int}\bigcap_{\nu\in\mathbb{N}}F(M_{\nu}^{*}AM_{\nu})}\subseteq\overline{\Lambda_{k}(A)}=\Lambda_{k}(A).
\end{equation}
By \eqref{eq4}, \eqref{eq5} and \eqref{eq8}, the required equality is asserted.

Consider now that $\Lambda_{k}(A)$ has no interior points, namely, it is a line segment or a singleton.
Then there is a suitable affine subspace $\mathcal{V}$ of $\mathbb{C}$ such that $\Lambda_{k}(A)\subseteq\mathcal{V}$ and with respect to the subspace topology, we have $\mathrm{int}\Lambda_{k}(A)\neq\emptyset$ and $\mathcal{V}\setminus\Lambda_{k}(A)$ be separable.
Following the same arguments as above, let $\widetilde{\mathcal{Z}}$ be a countable dense subset of $\mathcal{V}\setminus\Lambda_{k}(A)$. Hence, there is a sequence $\{\widetilde{z}_{q}: q\in\mathbb{N}\}$ in $\widetilde{\mathcal{Z}}$ converging to $z_{0}$ and  $\widetilde{z}_{q}\in\mathcal{B}(z_{0},\varepsilon)\subset
\bigcap_{\nu\in\mathbb{N}}F(M_{\nu}^{*}AM_{\nu})$. On the other hand, by \eqref{eq0}, we have $\widetilde{z}_{q}\notin \bigcap_{q\in\mathbb{N}}F(M_{i_{q}}^{*}AM_{i_{q}})$ for some indices $i_{q}\in\mathbb{N}$.
Clearly, we are led to a contradiction and we deduce $\bigcap_{\nu\in\mathbb{N}}F(M_{\nu}^{*}AM_{\nu})\subseteq\Lambda_{k}(A)$. Hence, with \eqref{eq1}, we conclude
\[
\Lambda_{k}(A)=\bigcap_{\nu\in\mathbb{N}}F(M_{\nu}^{*}AM_{\nu}).
\]
\end{proof}
{\bf Acknowledgement.} The authors would like to express their thanks to the reviewer for his comment on Theorem \ref{th2}.

\bibliographystyle{amsplain}

\end{document}